\documentclass[11pt]{article}%
\usepackage{amssymb}
\usepackage{amsmath}
\usepackage{amsfonts}
\usepackage{graphicx}%
\newtheorem{theorem}{Theorem}

\newtheorem{definition}[theorem]{Definition}
\newtheorem{example}[theorem]{Example}

\newtheorem{lemma}[theorem]{Lemma}

\newtheorem{proposition}[theorem]{Proposition}

\newenvironment{proof}[1][Proof]{\noindent\textbf{#1.} }{\ \rule{0.5em}{0.5em}}
\begin{document}

\title{The Similarity Invariants of non-lightlike curves in the Minkowski 3-space}
\author{Hakan \c{S}im\c{s}ek
\and Mustafa \"{O}zdemir}
\maketitle

\begin{abstract}
In this paper, we firstly introduce the group of similarity transformations in
the Minkowski-3 space. We describe differential- geometric invariants of a
non-lightlike curve according to the group of similarity transformations of
the Minkowski 3-space. We show extension of fundamental theorem for
non-lightlike curves under the group of similarity of the Minkowski 3-space.

\qquad\ 



\end{abstract}

\section{\textbf{Introduction}}

\qquad A similarity transformation (or similitude) of Euclidean space, which
consists of a rotation, a translation and an isotropic scaling, is an
automorphism preserving the angles and ratios between lengths. The geometric
properties unchanged by similarity transformations is called the
\textit{similarity geometry}. The whole Euclidean geometry can be considered
as a glass of similarity geometry. The similarity transformations are studying
in most area of the pure and applied mathematics.

\qquad Curve matching is an important research area in the computer vision and
pattern recognition, which can help us determine what category the given test
curve belongs to. Also, the recognition and pose determination of 3D objects
can be represented by space curves is important for industry automation,
robotics, navigation and medical applications. S. Li \cite{vision2} showed an
invariant representation based on so-called similarity-invariant coordinate
system (SICS) for matching 3D space curves under the group of similarity
transformations. He also \cite{vision3} presented a system for matching and
pose estimation of 3D space curves under the similarity transformation. Brook
et al. \cite{vision00} discussed various problems of image processing and
analysis by using the similarity transformation. Sahbi \cite{vision0}
investigated a method for shape description based on kernel principal
component analysis (KPCA) in the similarity invariance of KPCA. There are many
applications of the similarity transformation in the computer vision and
pattern recognition (see also \cite{vision01, vision1}).

\qquad The idea of self-similarity is one of the most basic and fruitful ideas
in mathematics. A self-similar object is exactly similar to a part of itself,
which in turn remains similar to a smaller part of itself, and so on. In the
last few decades it established itself as the central notion in areas such as
fractal geometry, dynamical systems, computer networks and statistical
physics. Mandelbrot presented the first description of self-similar sets,
namely sets that may be expressed as unions of rescaled copies of themselves.
He called these sets fractals, which are systems that present such
self-similar behavior and the examples in nature are many. The Cantor set, the
von Koch snowflake curve and the Sierpinski gasket are some of the most famous
examples of such sets. Hutchinson and, shortly thereafter, Barnsley and Demko
showed how systems of contractive maps with associated probabilities, referred
to as Iterated Function Systems (IFS), can be used to construct fractal,
self-similar sets and measures supported on such sets (see \cite{fractal,
fractal1, fractal2, fractal3, fractal30}).

\qquad When Euclidean 3-space is endowed with Lorentzian inner product, we
obtain \textit{Lorentzian similarity geometry. }Lorentzian flat
geometry\textit{ }is inside the\textit{ }Lorentzian similarity
geometr\textit{y}.\textit{\ }Kamishima \cite{simlorentz}\textit{ }studied the
properties of compact Lorentzian similarity manifolds using developing maps
and holonomy representations. The geometric invariants of curves in the
Lorentzian similarity geometry have not been considered so far. The theme of
similarity and self-similarity will be interesting in the Lorentzian-Minkowski space.

\qquad Many integrable equations, like Korteweg-de Vries (mKdV), sine-Gordon
and nonlinear Schr\"{o}dinger (NLS) equations, in soliton theory have been
shown to be related to motions of inextensible curves in the Euclidean space.
By using the similarity invariants of curves under the similarity motion, KS.
Chou and C. Qu \cite{chaos} showed that the motions of curves in two-, three-
and n-dimensional $(n>3)$ similarity geometries correspond to the Burgers
hierarchy, Burgers-mKdV hierarchy and a multi-component generalization of
these hierarchies in $\mathbb{E}^{n}$. Moreover, to study the motion of curves
in the Minkowski space also attracted researchers' interest. G\"{u}rses
\cite{gurses} studied the motion of curves on two-dimensional surface in
Minkowski 3-space. Q. Ding and J. Inoguchi \cite{inoguchi} showed that
binormal motions of curves in Minkowski 3-space are equivalent to some
integrable equations Therefore, the current paper will contribute to study the
motion of curves with similarity invariants in $\mathbb{E}_{1}^{3}$.

\qquad The broad content of similarity transformations were given by
\cite{berger} in arbitrary-dimensional Euclidean spaces. Differential
geometric invariants of Frenet curves up to the group of similarities were
studied by \cite{shape} in the Euclidean 3-space. In current paper, Lorentzian
version of similarity transformations will be entitled by pseudo-similarity
transformation defined by $\left(  \ref{05}\right)  $ in the section 2. The
main idea of this paper is to extend the fundamental theorem for a non-null
curve with respect to p-similarity motion and determine non-null self-similar
curves in the Minkowski 3-space.

\qquad The content of paper is as follows. We prove that p-similarity
transformations preserve the causal characters of vectors and the angles in
$\mathbb{E}_{1}^{3}$. We examine invariants of a non-lightlike Frenet curve up
to the group of p-similarities. We also show the relationship between the
focal curvatures and these invariants for non-lightlike Frenet curves in
$\mathbb{E}_{1}^{3}$. We give the uniqueness theorem which states that two
non-lightlike Frenet curves having same the p-shape curvature and same the
p-shape torsion are equivalent modulo a p-similarity. Furthermore, we obtain
the existence theorem that is a procedure for construction of a non-lightlike
Frenet curve by means of its p-shape curvature and p-shape torsion under some
initial conditions. Lastly, we give examples about construction of a
non-lightlike Frenet curve with a given p-shape.

\section{The Fundamental Group of Lorentzian Similarity Geometry}

\qquad Firstly, let us give some basic notions of the Lorentzian geometry. Let
$\mathbf{x}=\left(  x_{1},x_{2},x_{3}\right)  ^{T},$ $\mathbf{y}=\left(
y_{1},y_{2},y_{3}\right)  ^{T}$ and $\mathbf{z}=\left(  z_{1},z_{2}%
,z_{3}\right)  ^{T}$ be three arbitrary vectors in the Minkowski space
$\mathbb{E}_{1}^{3}.$ The Lorentzian inner product of $\mathbf{x}$ and
$\mathbf{y}$ can be stated as $\mathbf{x}\cdot\mathbf{y}=\mathbf{x}^{T}%
I^{\ast}\mathbf{y}$ where $I^{\ast}=diag(-1,1,1).$ The vector $\mathbf{x}$ in
$\mathbb{E}_{1}^{3}$ is called a spacelike vector, lightlike (or null) vector
and timelike vector if $\mathbf{x}\cdot\mathbf{x}>0$ or $\mathbf{x}=0,$
$\mathbf{x}\cdot\mathbf{x}=0$ or $\mathbf{x}\cdot\mathbf{x}<0,$ respectively.
The norm of the vector $\mathbf{x}$ is described by $\left\Vert \mathbf{x}%
\right\Vert =\sqrt{\left\vert \mathbf{x}\cdot\mathbf{x}\right\vert }.$ The
Lorentzian vector product $\mathbf{x}\times\mathbf{y}$ of $\mathbf{x}$ and
$\mathbf{y}$ is defined as follows:%
\[
\mathbf{x}\times\mathbf{y=}%
\begin{bmatrix}
-i & j & k\\
x_{1} & x_{2} & x_{3}\\
y_{1} & y_{2} & y_{3}%
\end{bmatrix}
\]
The hyperbolic and Lorentzian unit spheres are
\[
H_{0}^{2}=\left\{  \mathbf{x}\in\mathbb{E}_{1}^{3}:\mathbf{x\cdot
x=}-1\right\}  \text{ and }S_{1}^{2}=\left\{  \mathbf{x}\in\mathbb{E}_{1}%
^{3}:\mathbf{x\cdot x=}1\right\}
\]
respectively. There are two components $H_{0}^{2}$ passing through $\left(
1,0,0\right)  $ and $\left(  -1,0,0\right)  $ a future pointing hyperbolic
unit sphere a past pointing hyperbolic unit sphere, and they are denoted by
$H_{0}^{2+}$ and $H_{0}^{2-}$, respectively (see \cite{semi riemann} and
\cite{grub}).

\begin{theorem}
\label{min}Let $\mathbf{x}$ and $\mathbf{y}$ be vectors in the Minkowski 3-space.

$\left(  i\right)  $ If $\mathbf{x}$ and $\mathbf{y}$ are future-pointing (or
past-pointing) timelike vectors, then $\mathbf{x}\times\mathbf{y}$ is a
spacelike vector, $\mathbf{x\cdot y}=-\left\Vert \mathbf{x}\right\Vert
\left\Vert \mathbf{y}\right\Vert \cosh\theta$ and $\mathbf{x}\times
\mathbf{y=}\left\Vert \mathbf{x}\right\Vert \left\Vert \mathbf{y}\right\Vert
\sinh\theta$ where $\theta$ is the hyperbolic angle between $\mathbf{x}$ and
$\mathbf{y.}$

$\left(  ii\right)  $ If $\mathbf{x}$ and $\mathbf{y}$ are spacelike vectors
satisfying the inequality $\left\vert \mathbf{x\cdot y}\right\vert <\left\Vert
\mathbf{x}\right\Vert \left\Vert \mathbf{y}\right\Vert ,$ then $\mathbf{x}%
\times\mathbf{y}$ is timelike, $\mathbf{x\cdot y}=\left\Vert \mathbf{x}%
\right\Vert \left\Vert \mathbf{y}\right\Vert \cos\theta$ and $\mathbf{x}%
\times\mathbf{y=}\left\Vert \mathbf{x}\right\Vert \left\Vert \mathbf{y}%
\right\Vert \sin\theta$ where $\theta$ is the angle between $\mathbf{x}$ and
$\mathbf{y.}$

$\left(  iii\right)  $ If $\mathbf{x}$ and $\mathbf{y}$ are spacelike vectors
satisfying the inequality $\left\vert \mathbf{x\cdot y}\right\vert >\left\Vert
\mathbf{x}\right\Vert \left\Vert \mathbf{y}\right\Vert ,$ then $\mathbf{x}%
\times\mathbf{y}$ is spacelike, $\mathbf{x\cdot y}=\left\Vert \mathbf{x}%
\right\Vert \left\Vert \mathbf{y}\right\Vert \cosh\theta$ and $\mathbf{x}%
\times\mathbf{y=}\left\Vert \mathbf{x}\right\Vert \left\Vert \mathbf{y}%
\right\Vert \sinh\theta$ where $\theta$ is the hyperbolic angle between
$\mathbf{x}$ and $\mathbf{y.}$

$\left(  iv\right)  $ If $\mathbf{x}$ and $\mathbf{y}$ are spacelike vectors
satisfying the equality $\left\vert \mathbf{x\cdot y}\right\vert =\left\Vert
\mathbf{x}\right\Vert \left\Vert \mathbf{y}\right\Vert ,$ then $\mathbf{x}%
\times\mathbf{y}$ is lightlike.
\end{theorem}

\qquad Now, we define similarity transformation in $\mathbb{E}_{1}^{3}%
.$\ A\emph{ pseudo-similarity (in short p-similarity)} of Minkowski 3-space
$\mathbb{E}_{1}^{3}$ is a decomposition of a homothety (dilatation) a
pseudo-orthogonal map and a translation. Let $\mathbb{\hat{H}}$ be the split
quaternion algebra and $\mathbb{T\hat{H}}$ be the set of timelike split
quaternions such that we identify $\mathbb{E}_{1}^{3}$ with Im$\mathbb{\hat
{H}}$. $\mathbb{T\hat{H}}$ forms a group under the split quaternion product. A
unit timelike split quaternion represents a rotation in the Minkowski 3-space.
Therefore, by \cite{rotation}, there exists a unit timelike split quaternion
$q$ such that the transformation $\mathbf{R}_{q}:$ Im$\mathbb{T\hat
{H}\rightarrow}$ Im$\mathbb{T\hat{H}}$ defined by
\[
\mathbf{R}_{q}\left(  r\right)  =qrq^{-1}%
\]
can interpret rotation of a vector in the Minkowski 3-space. Thus, we get
\begin{equation}
f\left(  r\right)  =\mu qrq^{-1}+\mathbf{b}\label{05}%
\end{equation}
for some fixed $\mu\neq0\in%
\mathbb{R}
$ and $\mathbf{b}\in$Im$\mathbb{\hat{H}\cong}\mathbb{E}_{1}^{3}$. Since $f$ is
a affine map, we get $\left\Vert \vec{f}\left(  \mathbf{u}\right)  \right\Vert
=\left\vert \mu\right\vert \left\Vert \mathbf{u}\right\Vert $ for any
$\mathbf{u}\in\mathbb{E}_{1}^{3}$ where $\vec{f}\left(  \overrightarrow
{xy}\right)  =\overrightarrow{f(x)f(y)}$ (see \cite{berger}). The constant
$\left\vert \mu\right\vert $ is called a p-similarity ratio of the
transformation $f$. The p-similarity transformations are a group under the
composition of maps and denoted by \textbf{Sim}$\left(  \mathbb{E}_{1}%
^{3}\right)  $. This group is a fundamental group of the Lorentzian similarity
geometry. Also, the group of orientation-preserving (reversing) p-similarities
are denoted by \textbf{Sim}$^{+}\left(  \mathbb{E}_{1}^{3}\right)  $
(\textbf{Sim}$^{-}\left(  \mathbb{E}_{1}^{3}\right)  ,$ resp. ).

\begin{theorem}
The p-similarity transformations preserve the causal characters and angles.
\end{theorem}

\begin{proof}
Let $f$ be a p-similarity. Then, since we can write the equation
\begin{equation}
\vec{f}(\mathbf{u})\cdot\vec{f}(\mathbf{u})=\mu^{2}\left(  \mathbf{u\cdot
u}\right)  , \label{06}%
\end{equation}
$f$ preserves the causal character in $\mathbb{E}_{1}^{3}.$

\qquad Let $\mathbf{u}$ and $\mathbf{v}$ be future-pointing (or past-pointing)
timelike vectors and $\theta,$ $\gamma$ be the angle between $\mathbf{u}$,
$\mathbf{v}$ and $\vec{f}(\mathbf{u)}$, $\vec{f}(\mathbf{v)}$ respectively.
Since $\vec{f}(\mathbf{u)}$ and $\vec{f}(\mathbf{v)}$ have same causal
characters with $\mathbf{u}$ and $\mathbf{v,}$ we can find the following
equation from Theorem $\ref{min};$%
\begin{align}
\vec{f}(\mathbf{u})\cdot\vec{f}(\mathbf{v})  &  =-\left\Vert \vec
{f}(\mathbf{u})\right\Vert \left\Vert \vec{f}(\mathbf{v})\right\Vert
\cosh\gamma\label{011}\\
\mu^{2}\left(  \mathbf{u\cdot v}\right)   &  =-\mu^{2}\left\Vert
\mathbf{u}\right\Vert \left\Vert \mathbf{v}\right\Vert \cosh\gamma\nonumber\\
-\left\Vert \mathbf{u}\right\Vert \left\Vert \mathbf{v}\right\Vert \cosh\theta
&  =-\left\Vert \mathbf{u}\right\Vert \left\Vert \mathbf{v}\right\Vert
\cosh\gamma\nonumber\\
\cosh\theta &  =\cosh\gamma.\nonumber
\end{align}
From here, we have $\theta=\gamma.$ If $\mathbf{u}$ and $\mathbf{v}$ are
spacelike vectors satisfying the inequality $\left\vert \mathbf{u\cdot
v}\right\vert <\left\Vert \mathbf{u}\right\Vert \left\Vert \mathbf{v}%
\right\Vert ,$ then
\[
\left\Vert \vec{f}(\mathbf{u})\right\Vert \left\Vert \vec{f}(\mathbf{v}%
)\right\Vert =\mu^{2}\left\Vert \mathbf{u}\right\Vert \left\Vert
\mathbf{v}\right\Vert >\mu^{2}\left\vert \mathbf{u\cdot v}\right\vert
=\left\vert \vec{f}(\mathbf{u})\cdot\vec{f}(\mathbf{v})\right\vert .
\]
Therefore, it can be said from Theorem $\ref{min}$ that we have $\theta
=\gamma$ similar to $\left(  \ref{011}\right)  .$

\qquad It can also be found that $\theta$ is equal to $\gamma$ in case of
condition $\left(  iii\right)  $ in the Theorem $\ref{min}.$ As a consequence,
Every p-similarity transformation preserves the angle between any two vectors.
\end{proof}

\section{Geometric Invariants of non-lightlike Curves in the Lorentzian
Similarity Geometry}

\qquad Let $\alpha:t\in I\rightarrow\alpha\left(  t\right)  \in\mathbb{E}%
_{1}^{3}$ be a non-lightlike curve of class $C^{3}$ and $\kappa_{\alpha}$ and
$\tau_{\alpha}$ show curvature and torsion of $\alpha$, respectively. We
denote image of $\alpha$ under $f\in$ \textbf{Sim}$\left(  \mathbb{E}_{1}%
^{3}\right)  $ by $\beta.$ Then $\beta$ can be stated as
\begin{equation}
\beta\left(  t\right)  =\mu q\alpha\left(  t\right)  q^{-1}+\mathbf{b}%
\in\text{Im}\mathbb{\hat{H}},\text{ \ \ \ \ \ \ }t\in I.\label{1}%
\end{equation}

\qquad The arc length functions of $\alpha$ and $\beta$ starting at $t_{0}\in
I$ are%
\begin{equation}
s(t)=\int\limits_{t_{0}}^{t}\left\Vert \frac{d\alpha\left(  u\right)  }%
{du}\right\Vert du,\text{ \ \ \ \ \ \ }s^{\ast}\left(  t\right)
=\int\limits_{t_{0}}^{t}\left\Vert \frac{d\beta\left(  u\right)  }%
{du}\right\Vert du=\left\vert \mu\right\vert s\left(  t\right)  .\label{1'}%
\end{equation}
The Frenet-Serret formulas of $\alpha$ in the Minkowski 3-space is
\begin{equation}
\frac{d}{ds}%
\begin{bmatrix}
\mathbf{e}_{1}\\
\mathbf{e}_{2}\\
\mathbf{e}_{3}%
\end{bmatrix}
=%
\begin{bmatrix}
0 & \kappa_{\alpha} & 0\\
\varepsilon_{\mathbf{e}_{3}}\kappa_{\alpha} & 0 & \tau_{\alpha}\\
0 & \varepsilon_{\mathbf{e}_{1}}\tau_{\alpha} & 0
\end{bmatrix}%
\begin{bmatrix}
\mathbf{e}_{1}\\
\mathbf{e}_{2}\\
\mathbf{e}_{3}%
\end{bmatrix}
\label{0}%
\end{equation}
where $\left\{  \mathbf{e}_{1},\mathbf{e}_{2},\mathbf{e}_{3}\right\}  $ is
Frenet frame of $\alpha$ and $\varepsilon_{\mathbf{e}_{\ell}}=\mathbf{e}%
_{\ell}\cdot\mathbf{e}_{\ell}\ $for $1\leq\ell\leq3.$ (see \cite{biharmonic}
and \cite{minkowski}). In this section, the differentiation according to $s$
is denoted by primes. The curvature $\kappa_{\alpha}$ and torsion
$\tau_{\alpha}$ of non-lightlike curve $\alpha$ is given by
\begin{equation}
\kappa_{\alpha}\left(  s\right)  =\left\Vert \alpha^{\prime}\times
\alpha^{\prime\prime}\right\Vert ,\text{ \ \ \ \ \ \ }\tau_{\alpha}\left(
s\right)  =\frac{\det\left(  \alpha^{\prime},\alpha^{\prime\prime}%
,\alpha^{\prime\prime\prime}\right)  }{\left\Vert \alpha^{\prime}\times
\alpha^{\prime\prime}\right\Vert ^{2}}.\label{2}%
\end{equation}
From $\left(  \ref{1}\right)  ,$ $\left(  \ref{1'}\right)  $ and $\left(
\ref{2}\right)  ,$ we can calculate the curvature $\kappa_{\beta}\left(
\left\vert \mu\right\vert s\right)  $ and the torsion $\tau_{\beta}\left(
\left\vert \mu\right\vert s\right)  $ as
\begin{equation}
\kappa_{\beta}=\left\Vert \beta^{\prime}\times\beta^{\prime\prime}\right\Vert
=\frac{1}{\left\vert \mu\right\vert }\kappa_{\alpha}\left(  s\right)
\label{02}%
\end{equation}
and
\begin{equation}
\tau_{\beta}=\dfrac{1}{\mu}\tau_{\alpha}\left(  s\right)  .\label{03}%
\end{equation}
Since we have $ds^{\ast}=\left\vert \mu\right\vert ds$ from $\left(
\ref{1'}\right)  ,$ we get $\kappa_{\alpha}ds=\kappa_{\beta}ds^{\ast}$ and
$\left\vert \tau_{\alpha}\right\vert ds=\left\vert \tau_{\beta}\right\vert
ds^{\ast}.$

\qquad Let $\sigma_{\alpha}$ and $\sigma_{\beta}$ be spherical arc-length
parameters of $\alpha$ and $\beta,$ respectively. Then, we can find that
\begin{equation}
d\sigma_{\alpha}=\kappa_{\alpha}ds=\kappa_{\beta}ds^{\ast}=d\sigma_{\beta
}.\label{3}%
\end{equation}
Thus, the spherical arc-length element $d\sigma_{\alpha}$ is invariant under
the group of the p-similarities of $\mathbb{E}_{1}^{3}.$ The derivative
formulas of $\alpha$ with respect to $\sigma_{\alpha}$ are given by
\begin{equation}
\frac{d\alpha}{d\sigma_{\alpha}}=\frac{1}{\kappa_{\alpha}}\mathbf{e}%
_{1},\text{ \ \ \ \ \ \ \ }\frac{d^{2}\alpha}{d\sigma_{\alpha}^{2}}%
=-\frac{d\kappa_{\alpha}}{\kappa_{\alpha}d\sigma_{\alpha}}\frac{d\alpha
}{d\sigma_{\alpha}}+\frac{1}{\kappa_{\alpha}}\mathbf{e}_{2}\label{4}%
\end{equation}
and%
\begin{equation}
\frac{d}{d\sigma_{\alpha}}%
\begin{bmatrix}
\mathbf{e}_{1}\\
\mathbf{e}_{2}\\
\mathbf{e}_{3}%
\end{bmatrix}
=%
\begin{bmatrix}
0 & 1 & 0\\
\varepsilon_{3} & 0 & \frac{\tau_{\alpha}}{\kappa_{\alpha}}\\
0 & \varepsilon_{1}\frac{\tau_{\alpha}}{\kappa_{\alpha}} & 0
\end{bmatrix}%
\begin{bmatrix}
\mathbf{e}_{1}\\
\mathbf{e}_{2}\\
\mathbf{e}_{3}%
\end{bmatrix}
\label{6}%
\end{equation}
by means of $\left(  \ref{0}\right)  $ and $\left(  \ref{3}\right)  .$
Similarly, for the non-lightlike curve $\beta$ we also have
\begin{equation}
\frac{d^{2}\beta}{d\sigma_{\beta}^{2}}=-\frac{d\kappa_{\beta}}{\kappa_{\beta
}d\sigma_{\beta}}\frac{d\beta}{d\sigma_{\beta}}+\frac{1}{\kappa_{\beta}%
}\mathbf{e}_{2}^{\ast}\label{7}%
\end{equation}
where $\left\{  \mathbf{e}_{1}^{\ast},\mathbf{e}_{2}^{\ast},\mathbf{e}%
_{3}^{\ast}\right\}  $ is a Frenet frame field along the non-lightlike curve
$\beta.$ From $\left(  \ref{02}\right)  ,$ $\left(  \ref{03}\right)  $ and
$\left(  \ref{3}\right)  ,$ we can write%
\[
-\frac{d\kappa_{\beta}}{\kappa_{\beta}d\sigma_{\beta}}=-\frac{d\kappa_{\alpha
}}{\kappa_{\alpha}d\sigma_{\alpha}}\text{ \ \ and \ \ }\frac{\tau_{\beta}%
}{\kappa_{\beta}}=\frac{\left\vert \mu\right\vert }{\mu}\frac{\tau_{\alpha}%
}{\kappa_{\alpha}}.
\]
If we take $\mu>0,$ i.e. the p-similarity is an orientation-preserving
transformation, we get $\dfrac{\tau_{\beta}}{\kappa_{\beta}}=\dfrac
{\tau_{\alpha}}{\kappa_{\alpha}}.$ Thus, we obtain the following Lemma from
above calculations.

\begin{lemma}
The functions $\tilde{\kappa}_{\alpha}=-\dfrac{d\kappa_{\alpha}}%
{\kappa_{\alpha}d\sigma_{\alpha}}$ and $\tilde{\tau}_{\alpha}=\dfrac
{\tau_{\alpha}}{\kappa_{\alpha}}$ are invariants under the group of the
orientation-preserving p-similarities of the Minkowski 3-space.
\end{lemma}

\qquad Using $\left(  \ref{4}\right)  $ and $\left(  \ref{6}\right)  $ the
invariants $\tilde{\kappa}_{\alpha}$ and $\tilde{\tau}_{\alpha}$ can take the
form
\begin{equation}
\tilde{\kappa}_{\alpha}\left(  \sigma_{\alpha}\right)  =\dfrac{\frac
{d^{2}\alpha}{d\sigma_{\alpha}^{2}}\cdot\frac{d\alpha}{d\sigma_{\alpha}}%
}{\frac{d\alpha}{d\sigma_{\alpha}}\cdot\frac{d\alpha}{d\sigma_{\alpha}}%
},\label{8}%
\end{equation}%
\begin{equation}
\tilde{\tau}_{\alpha}\left(  \sigma_{\alpha}\right)  =\det\left(
\frac{d\alpha}{d\sigma_{\alpha}},\frac{d^{2}\alpha}{d\sigma_{\alpha}^{2}%
},\frac{d^{3}\alpha}{d\sigma_{\alpha}^{3}}\right)  \frac{\left\Vert
\frac{d\alpha}{d\sigma_{\alpha}}\right\Vert ^{3}}{\left\Vert \frac{d\alpha
}{d\sigma_{\alpha}}\times\frac{d^{2}\alpha}{d\sigma_{\alpha}^{2}}\right\Vert
^{3}}.\label{9}%
\end{equation}

\begin{definition}
Let $\alpha:I\rightarrow\mathbb{E}_{1}^{3}$ be a non-lightlike Frenet curve of
the class $C^{3}$ parameterized by the spherical arc length parameter
$\sigma_{\alpha}.$ Let $\kappa_{\alpha}\left(  \sigma_{\alpha}\right)  $ and
$\tau_{\alpha}\left(  \sigma_{\alpha}\right)  $ be the curvature and torsion
of $\alpha,$ respectively. The functions
\begin{equation}
\tilde{\kappa}_{\alpha}=-\dfrac{d\kappa_{\alpha}}{\kappa_{\alpha}%
d\sigma_{\alpha}}\text{ \ \ \ and \ \ \ \ }\tilde{\tau}_{\alpha}=\dfrac
{\tau_{\alpha}}{\kappa_{\alpha}}\label{000}%
\end{equation}
are \emph{p-shape curvature }and \emph{p-shape torsion }of $\alpha.$ The
ordered pair $\left(  \tilde{\kappa}_{\alpha},\tilde{\tau}_{\alpha}\right)  $
is called a (local) \emph{p-shape }of the non-lightlike curve $\alpha$ in the
Minkowski 3-space.
\end{definition}

\qquad We consider the pseudo-orthogonal 3-frame $\left\{  \mathbf{e}%
_{1}\left(  \sigma_{\alpha}\right)  /\kappa_{\alpha},\mathbf{e}_{2}\left(
\sigma_{\alpha}\right)  /\kappa_{\alpha},\mathbf{e}_{3}\left(  \sigma_{\alpha
}\right)  /\kappa_{\alpha}\right\}  ,$ $\sigma_{\alpha}\in I,$ for the curve.
Then, by the equations $\left(  \ref{4}\right)  $ and $\left(  \ref{6}\right)
,$ we get
\begin{equation}
\frac{d}{d\sigma_{\alpha}}%
\begin{bmatrix}
\mathbf{e}_{1}/\kappa_{\alpha}\\
\mathbf{e}_{2}/\kappa_{\alpha}\\
\mathbf{e}_{3}/\kappa_{\alpha}%
\end{bmatrix}
=%
\begin{bmatrix}
\tilde{\kappa}_{\alpha} & 1 & 0\\
\varepsilon_{\mathbf{e}_{3}} & \tilde{\kappa}_{\alpha} & \tilde{\tau}_{\alpha
}\\
0 & \varepsilon_{\mathbf{e}_{1}}\tilde{\tau}_{\alpha} & \tilde{\kappa}%
_{\alpha}%
\end{bmatrix}%
\begin{bmatrix}
\mathbf{e}_{1}/\kappa_{\alpha}\\
\mathbf{e}_{2}/\kappa_{\alpha}\\
\mathbf{e}_{3}/\kappa_{\alpha}%
\end{bmatrix}
.\label{fu}%
\end{equation}
The pseudo-orthogonal frame $\mathbf{e}_{1}\left(  \sigma_{\alpha}\right)
/\kappa_{\alpha},$ $\mathbf{e}_{2}\left(  \sigma_{\alpha}\right)
/\kappa_{\alpha},$ $\mathbf{e}_{3}\left(  \sigma_{\alpha}\right)
/\kappa_{\alpha}$ is invariant under the group \textbf{Sim}$^{+}\left(
\mathbb{E}_{1}^{3}\right)  .$ Thus, it may be said that the equation $\left(
\ref{fu}\right)  $ is the Frenet-Serret frame of $\alpha$ in the Lorentzian
similarity 3-space.

\subsection{The relation between focal curvatures and p-shape of $\alpha$}

\qquad Let $\alpha:I\rightarrow\mathbb{E}_{1}^{3}$ be a unit speed
non-lightlike Frenet curve with the Frenet frame $\mathbf{e}_{1},$
$\mathbf{e}_{2},$ $\mathbf{e}_{3}$ and let $s$ be an arc length parameter of
$\alpha.$ The curve $\mathbf{\gamma:}I\rightarrow\mathbb{E}_{1}^{3}$
consisting of the centers of the osculating sphere of the curve $\alpha$ is
called the \emph{focal curve }of $\alpha.$ The focal curve can be represented
by
\[
\mathbf{\gamma}\left(  s\right)  =\alpha\left(  s\right)  +m_{1}\left(
s\right)  \mathbf{e}_{2}+m_{2}\left(  s\right)  \mathbf{e}_{3}%
\]
where $m_{1}$ and $m_{2}$ are smooth functions called \emph{focal curvature
}of $\alpha.$ Then, we have the following theorem from \cite{minkowski}.

\begin{theorem}
\label{focal}Let $\alpha$ be a non-lightlike curve in $\mathbb{E}_{1}^{3},$
the radius and center of the osculating sphere of $\alpha$ at $\alpha\left(
s\right)  $ are
\[
r=\sqrt{\left(  \varepsilon_{\mathbf{e}_{2}}\right)  \frac{1}{\kappa^{2}%
}+\left(  \varepsilon_{\mathbf{e}_{3}}\right)  \left(  \frac{\kappa^{^{\prime
}}}{\kappa^{2}\tau}\right)  }\text{ \ and \ }\mathbf{\gamma}\left(  s\right)
=\alpha\left(  s\right)  +\frac{\varepsilon_{\mathbf{e}_{1}}\varepsilon
_{\mathbf{e}_{2}}}{\kappa}\mathbf{e}_{2}+\frac{\varepsilon_{\mathbf{e}_{1}%
}\varepsilon_{\mathbf{e}_{3}}}{\tau}\left(  \frac{1}{\kappa}\right)
^{^{\prime}}\mathbf{e}_{3}%
\]
where $\mathbf{e}_{2},$ and $\mathbf{e}_{3}$ are normal and binormal vector
fields of the curve at $\alpha\left(  s\right)  .$
\end{theorem}

\qquad Using the Theorem $\ref{focal}$ we state that the focal curvatures
$m_{1}$ and $m_{2}$ of the non-lightlike curve $\alpha$ are equal to
\begin{equation}
\dfrac{\varepsilon_{\mathbf{e}_{1}}\varepsilon_{\mathbf{e}_{2}}}%
{\kappa_{\alpha}}\text{ \ \ \ \ and \ \ \ \ }\dfrac{1}{\tau_{\alpha}}\left(
\dfrac{\varepsilon_{\mathbf{e}_{1}}\varepsilon_{\mathbf{e}_{3}}}%
{\kappa_{\alpha}}\right)  ^{^{\prime}}\label{001}%
\end{equation}
respectively. Now, we can show the relation between the focal curvatures and
the p-shape curvature and torsion.

\begin{proposition}
Let $\alpha:I\rightarrow\mathbb{E}_{1}^{3}$ be a unit speed non-lightlike
Frenet curve with the non-zero curvature $\kappa$ and torsion $\tau$. Then,
\[
\tilde{\kappa}_{\alpha}=\varepsilon_{\mathbf{e}_{1}}\varepsilon_{\mathbf{e}%
_{2}}m_{1}^{^{\prime}}\text{ \ \ \ \ and \ \ \ \ }\tilde{\tau}_{\alpha
}=\varepsilon_{\mathbf{e}_{1}}\varepsilon_{\mathbf{e}_{3}}\frac{m_{1}%
^{^{\prime}}m_{1}}{f_{2}}.
\]

\end{proposition}

\begin{proof}
From $\left(  \ref{000}\right)  $ and $\left(  \ref{001}\right)  $ we can
write
\[
\tilde{\kappa}_{\alpha}=-\dfrac{d\kappa_{\alpha}}{\kappa_{\alpha}%
d\sigma_{\alpha}}=-\frac{1}{\kappa_{\alpha}^{2}}\dfrac{d\kappa_{\alpha}}%
{ds}=\left(  \frac{1}{\kappa_{\alpha}}\right)  ^{^{\prime}}=\varepsilon
_{\mathbf{e}_{1}}\varepsilon_{\mathbf{e}_{2}}m_{1}^{^{\prime}}%
\]
and
\[
\tilde{\tau}_{\alpha}=\dfrac{\tau_{\alpha}}{\kappa_{\alpha}}=\frac{1}%
{\kappa_{\alpha}}\tau_{\alpha}=\varepsilon_{\mathbf{e}_{1}}\varepsilon
_{\mathbf{e}_{2}}m_{1}\frac{\varepsilon_{\mathbf{e}_{1}}\varepsilon
_{\mathbf{e}_{3}}}{m_{2}}\left(  \varepsilon_{\mathbf{e}_{1}}\varepsilon
_{\mathbf{e}_{2}}m_{1}\right)  ^{^{\prime}}=\varepsilon_{\mathbf{e}_{1}%
}\varepsilon_{\mathbf{e}_{3}}\frac{m_{1}^{^{\prime}}m_{1}}{m_{2}}.
\]

\end{proof}

\section{Uniqueness Theorem}

\qquad Two non-lightlike Frenet curves which have the same torsion and the
same positive curvature are always equivalent according to Lorentzian motion.
This notion can be extended under the group \textbf{Sim}$\left(
\mathbb{E}_{1}^{3}\right)  $ for the non-lightlike Frenet curves which have
the same p-shape torsion and p-shape curvature, in the Minkowski 3-space
$\mathbb{E}_{1}^{3}.$

\begin{theorem}
\label{teklik}(Uniqueness Theorem) Let $\alpha,\alpha^{\ast}:I\rightarrow
\mathbb{E}_{1}^{3}$ be two non-lightlike Frenet curves of class $C^{3}$
parameterized by the same spherical arc length parameter $\sigma$ and have the
same causal characters, where $I\subset%
\mathbb{R}
$ is an open interval. Suppose that $\alpha$ and $\alpha^{\ast}$ have the same
p-shape curvatures $\tilde{\kappa}=\tilde{\kappa}^{\ast}$ and the same p-shape
torsions $\tilde{\tau}=\tilde{\tau}^{\ast}$ for any $\sigma\in I.$

$\mathbf{i})$ If $\alpha,\alpha^{\ast}$ are timelike curves, there exists a
$f\in$\textbf{Sim}$^{+}\left(  \mathbb{E}_{1}^{3}\right)  $ such that
$\alpha^{\ast}=f\circ\alpha.$

$\mathbf{ii})$ If $\alpha,\alpha^{\ast}$ are spacelike curves, there exists a
$f\in$\textbf{Sim}$^{-}\left(  \mathbb{E}_{1}^{3}\right)  $ such that
$\alpha^{\ast}=f\circ\alpha.$
\end{theorem}

\begin{proof}
Let $\kappa,\kappa^{\ast}$ and $\tau,$ $\tau^{\ast}$ be the curvatures and the
torsions of the $\alpha,$ $\alpha^{\ast}.$ Since $\alpha$ and $\alpha^{\ast}$
have the same shape curvatures $\tilde{\kappa}=\tilde{\kappa}^{\ast},$ we
have
\[
\frac{d\kappa}{\kappa}=\frac{d\kappa^{\ast}}{\kappa^{\ast}}\text{ \ \ or
\ \ }\log\kappa=\log\kappa^{\ast}+\log\mu
\]
where $\mu\in%
\mathbb{R}
^{+}.$ Then, we find $\kappa=\mu\kappa^{\ast}$ for any $\sigma\in I.$ Using
$\tilde{\tau}=\tilde{\tau}^{\ast}$ we get $\tau=\mu\tau^{\ast}$ for any
$\sigma\in I.$ Let $\mathbf{e}_{i}$, $\mathbf{e}_{i}^{\ast},$ $i=1,2,3,$ be a
Frenet frame fields on $\alpha$, $\alpha^{\ast}$ and we choose any point
$\sigma_{0}\in I.$ There exists a Lorentzian motion $\varphi$ of
$\mathbb{E}_{1}^{3}$ such that
\[
\varphi\left(  \alpha\left(  \sigma_{0}\right)  \right)  =\alpha^{\ast}\left(
\sigma_{0}\right)  \text{ \ \ and \ \ }\varphi\left(  \mathbf{e}_{i}\left(
\sigma_{0}\right)  \right)  =-\varepsilon_{\mathbf{e}_{i}}\mathbf{e}_{i}%
^{\ast}\left(  \sigma_{0}\right)  \text{ for }i=1,2,3.
\]
Let's consider the function $\Psi:I\rightarrow%
\mathbb{R}
$ defined by
\[
\Psi\left(  \sigma\right)  =\left\Vert \varphi\left(  \mathbf{e}_{1}\left(
\sigma\right)  \right)  +\varepsilon_{\mathbf{e}_{1}}\mathbf{e}_{1}^{\ast
}\left(  \sigma\right)  \right\Vert ^{2}+\left\Vert \varphi\left(
\mathbf{e}_{2}\left(  \sigma\right)  \right)  +\varepsilon_{\mathbf{e}_{2}%
}\mathbf{e}_{2}^{\ast}\left(  \sigma\right)  \right\Vert ^{2}+\left\Vert
\varphi\left(  \mathbf{e}_{3}\left(  \sigma\right)  \right)  +\varepsilon
_{\mathbf{e}_{3}}\mathbf{e}_{3}^{\ast}\left(  \sigma\right)  \right\Vert ^{2}.
\]
Then
\begin{align*}
\frac{d\Psi}{d\sigma} &  =2\left(  \frac{d}{d\sigma}\varphi\left(
\mathbf{e}_{1}\left(  \sigma\right)  \right)  +\varepsilon_{\mathbf{e}_{1}%
}\frac{d}{d\sigma}\mathbf{e}_{1}^{\ast}\left(  \sigma\right)  \right)
\cdot\left(  \varphi\left(  \mathbf{e}_{1}\left(  \sigma\right)  \right)
+\varepsilon_{\mathbf{e}_{1}}\mathbf{e}_{1}^{\ast}\left(  \sigma\right)
\right)  \\
&  +2\left(  \frac{d}{d\sigma}\varphi\left(  \mathbf{e}_{2}\left(
\sigma\right)  \right)  +\varepsilon_{\mathbf{e}_{2}}\frac{d}{d\sigma
}\mathbf{e}_{2}^{\ast}\left(  \sigma\right)  \right)  \cdot\left(
\varphi\left(  \mathbf{e}_{2}\left(  \sigma\right)  \right)  +\varepsilon
_{\mathbf{e}_{2}}\mathbf{e}_{2}^{\ast}\left(  \sigma\right)  \right)  \\
&  +2\left(  \frac{d}{d\sigma}\varphi\left(  \mathbf{e}_{3}\left(
\sigma\right)  \right)  +\varepsilon_{\mathbf{e}_{3}}\frac{d}{d\sigma
}\mathbf{e}_{3}^{\ast}\left(  \sigma\right)  \right)  \cdot\left(
\varphi\left(  \mathbf{e}_{3}\left(  \sigma\right)  \right)  +\varepsilon
_{\mathbf{e}_{3}}\mathbf{e}_{3}^{\ast}\left(  \sigma\right)  \right)  .
\end{align*}
Using $\left\Vert \varphi\left(  \mathbf{e}_{i}\right)  \right\Vert
^{2}=\left\Vert \mathbf{e}_{i}\right\Vert ^{2}=\left\Vert \mathbf{e}_{i}%
^{\ast}\right\Vert ^{2}=1$ we can write
\begin{align*}
\frac{d\Psi}{d\sigma} &  =2\varepsilon_{\mathbf{e}_{1}}\left[  \left(
\varphi\left(  \frac{d}{d\sigma}\mathbf{e}_{1}\right)  \right)  \cdot
\mathbf{e}_{1}^{\ast}+\varphi\left(  \mathbf{e}_{1}\right)  \cdot\left(
\frac{d}{d\sigma}\mathbf{e}_{1}^{\ast}\right)  \right]  \\
&  +2\varepsilon_{\mathbf{e}_{2}}\left[  \left(  \varphi\left(  \frac
{d}{d\sigma}\mathbf{e}_{2}\right)  \right)  \cdot\mathbf{e}_{2}^{\ast}%
+\varphi\left(  \mathbf{e}_{2}\right)  \cdot\left(  \frac{d}{d\sigma
}\mathbf{e}_{2}^{\ast}\right)  \right]  \\
&  +2\varepsilon_{\mathbf{e}_{3}}\left[  \left(  \varphi\left(  \frac
{d}{d\sigma}\mathbf{e}_{3}\right)  \right)  \cdot\mathbf{e}_{3}^{\ast}%
+\varphi\left(  \mathbf{e}_{3}\right)  \cdot\left(  \frac{d}{d\sigma
}\mathbf{e}_{3}^{\ast}\right)  \right]  .
\end{align*}
From $\left(  \ref{6}\right)  ,$ we get
\begin{align*}
\frac{d\Psi}{d\sigma} &  =\left(  2\varepsilon_{\mathbf{e}_{1}}+2\varepsilon
_{\mathbf{e}_{2}}\varepsilon_{\mathbf{e}_{3}^{\ast}}\right)  \left[
\varphi\left(  \mathbf{e}_{2}\right)  \cdot\mathbf{e}_{1}^{\ast}\right]
+\left(  2\varepsilon_{\mathbf{e}_{1}}+2\varepsilon_{\mathbf{e}_{2}%
}\varepsilon_{\mathbf{e}_{3}}\right)  \left[  \varphi\left(  \mathbf{e}%
_{1}\right)  \cdot\mathbf{e}_{2}^{\ast}\right]  \\
&  +\left(  2\varepsilon_{\mathbf{e}_{2}}\tilde{\tau}+2\varepsilon
_{\mathbf{e}_{3}}\varepsilon_{\mathbf{e}_{1}^{\ast}}\tilde{\tau}^{\ast
}\right)  \left[  \varphi\left(  \mathbf{e}_{3}\right)  \cdot\mathbf{e}%
_{2}^{\ast}\right]  +\left(  2\varepsilon_{\mathbf{e}_{2}}\tilde{\tau}^{\ast
}+2\varepsilon_{\mathbf{e}_{3}}\varepsilon_{\mathbf{e}_{1}}\tilde{\tau
}\right)  \left[  \varphi\left(  \mathbf{e}_{2}\right)  \cdot\mathbf{e}%
_{3}^{\ast}\right]  .
\end{align*}
Since $\alpha$ and $\alpha^{\ast}$ have the same causal characters and
$\tilde{\tau}=\tilde{\tau}^{\ast},$ we can write
\begin{align*}
2\varepsilon_{\mathbf{e}_{1}}+2\varepsilon_{\mathbf{e}_{2}}\varepsilon
_{\mathbf{e}_{3}^{\ast}} &  =0,\text{ \ \ \ \ \ \ \ \ \ }2\varepsilon
_{\mathbf{e}_{1}}+2\varepsilon_{\mathbf{e}_{2}}\varepsilon_{\mathbf{e}_{3}%
}=0\\
2\varepsilon_{\mathbf{e}_{2}}\tilde{\tau}+2\varepsilon_{\mathbf{e}_{3}%
}\varepsilon_{\mathbf{e}_{1}^{\ast}}\tilde{\tau}^{\ast} &  =0,\text{
\ \ \ \ \ }2\varepsilon_{\mathbf{e}_{2}}\tilde{\tau}^{\ast}+2\varepsilon
_{\mathbf{e}_{3}}\varepsilon_{\mathbf{e}_{1}}\tilde{\tau}=0.
\end{align*}
Therefore, we find $\dfrac{d\Psi}{d\sigma}=0$ for any $\sigma\in I.$ On the
other hand, we know $\Psi\left(  \sigma_{0}\right)  =0$ and thus we have
$\Psi\left(  \sigma\right)  =0$ for any $\sigma\in I.$ As a result, we can say
that
\begin{equation}
\varphi\left(  \mathbf{e}_{i}\left(  \sigma\right)  \right)  =-\varepsilon
_{\mathbf{e}_{i}}\mathbf{e}_{i}^{\ast}\left(  \sigma\right)  ,\text{
\ \ \ \ \ }\forall\sigma\in I,\text{ \ \ }i=1,2,3.\label{10}%
\end{equation}

\qquad The map $g=\mu\varphi:\mathbb{E}_{1}^{3}\rightarrow\mathbb{E}_{1}^{3}$
is a p-similarity of $\mathbb{E}_{1}^{3}$. We examine an other function
$\Phi:I\rightarrow%
\mathbb{R}
$ such that
\[
\Phi\left(  \sigma\right)  =\left\Vert \frac{d}{d\sigma}g\left(  \alpha\left(
\sigma\right)  \right)  +\varepsilon_{\mathbf{e}_{1}}\frac{d}{d\sigma}%
\alpha^{\ast}\left(  \sigma\right)  \right\Vert ^{2}\text{ \ \ \ for }%
\forall\sigma\in I.
\]
Taking derivative of this function with respect to $\sigma$ we get
\begin{align*}
\frac{d\Phi}{d\sigma} &  =2g\left(  \frac{d^{2}\alpha}{d\sigma^{2}}\right)
\cdot g\left(  \frac{d\alpha}{d\sigma}\right)  +2\varepsilon_{\mathbf{e}_{1}%
}\left[  g\left(  \frac{d^{2}\alpha}{d\sigma^{2}}\right)  \cdot\frac
{d\alpha^{\ast}}{d\sigma}\right]  \\
&  +2\varepsilon_{\mathbf{e}_{1}}\frac{d^{2}\alpha^{\ast}}{d\sigma^{2}}\cdot
g\left(  \frac{d\alpha}{d\sigma}\right)  +2\left[  \frac{d^{2}\alpha^{\ast}%
}{d\sigma^{2}}\cdot\frac{d\alpha^{\ast}}{d\sigma}\right]  .
\end{align*}
Since the function $\varphi$ is linear map and we have $\left(  \ref{4}%
\right)  $ and $\left(  \ref{10}\right)  $, we can write
\[
\frac{d\Phi}{d\sigma}=2\varepsilon_{\mathbf{e}_{1}^{\ast}}\mu^{2}\frac
{\tilde{\kappa}}{\kappa^{2}}-2\varepsilon_{\mathbf{e}_{1}^{\ast}}\mu
\frac{\tilde{\kappa}}{\kappa\kappa^{\ast}}-2\varepsilon_{\mathbf{e}_{1}^{\ast
}}\mu\frac{\tilde{\kappa}^{\ast}}{\kappa\kappa^{\ast}}+2\varepsilon
_{\mathbf{e}_{1}^{\ast}}\frac{\tilde{\kappa}^{\ast}}{\left(  \kappa^{\ast
}\right)  ^{2}}.
\]
Using $\mu=\dfrac{\kappa}{\kappa^{\ast}},$ we have $\dfrac{d\Phi}{d\sigma}=0.$
Also, we can find
\[
\frac{d}{d\sigma}g\left(  \alpha\left(  \sigma_{0}\right)  \right)  =g\left(
\frac{1}{\kappa}\mathbf{e}_{1}\left(  \sigma_{0}\right)  \right)
=-\varepsilon_{\mathbf{e}_{1}}\frac{1}{\kappa^{\ast}}\mathbf{e}_{1}^{\ast
}\left(  \sigma_{0}\right)
\]
and we know
\[
\frac{d}{d\sigma}\alpha^{\ast}\left(  \sigma_{0}\right)  =\frac{1}%
{\kappa^{\ast}}\mathbf{e}_{1}^{\ast}\left(  \sigma_{0}\right)  .
\]
Then, we conclude that $\Phi\left(  \sigma_{0}\right)  =0.$ Hence,
$\Phi\left(  \sigma\right)  =0$ for $\forall\sigma\in I$. This means that
\[
\frac{d}{d\sigma}g\left(  \alpha\left(  \sigma\right)  \right)  =-\varepsilon
_{\mathbf{e}_{1}}\frac{d}{d\sigma}\alpha^{\ast}\left(  \sigma\right)
\]
or equivalently $\alpha^{\ast}\left(  \sigma\right)  =-\varepsilon
_{\mathbf{e}_{1}}g\left(  \alpha\left(  \sigma\right)  \right)  +\mathbf{b}$
where $\mathbf{b}$ is a constant vector. Then, the image of non-lightlike
curve $\alpha$ under the p-similarity $f=\vartheta\circ\left(  -\varepsilon
_{\mathbf{e}_{1}}g\right)  $, where $\vartheta:\mathbb{E}_{1}^{3}%
\rightarrow\mathbb{E}_{1}^{3}$ is a translation function determined by
$\mathbf{b}$, is the non-lightlike curve $\alpha^{\ast}.$ If the curves
$\alpha,\alpha^{\ast}$ are taken as the timelike curves, the p-similarity
transformation $f$ is an orientation-preserving transformation. Also, when the
curves $\alpha,\alpha^{\ast}$ are the spacelike curves, the p-similarity
transformation $f$ is an orientation-reversing transformation.
\end{proof}

\qquad Is it possible to say that two spacelike Frenet curves are equivalent
under orientation-preserving p-similarity? We can see the answer with the
following theorem.

\begin{theorem}
\label{space}Let $\alpha,\alpha^{\ast}:I\rightarrow\mathbb{E}_{1}^{3}$ be two
spacelike Frenet curves of class $C^{3}$ parameterized by the same spherical
arc length parameter $\sigma$, where $I\subset%
\mathbb{R}
$ is an open interval. Suppose that $\alpha$ and $\alpha^{\ast}$ have the same
p-shape curvature $\tilde{\kappa}=\tilde{\kappa}^{\ast}$ and $\tilde{\tau
}=-\tilde{\tau}^{\ast}$ for the p-shape torsions $\tilde{\tau},$ $\tilde{\tau
}^{\ast}$. Then there exists an orientation-preserving p-similarity $f$ of
$\mathbb{E}_{1}^{3}$ such that $\alpha^{\ast}=f\circ\alpha.$
\end{theorem}

\begin{proof}
The proof is similar to the proof of the Theorem \ref{teklik}. Let
$\mathbf{e}_{i}$, $\mathbf{e}_{i}^{\ast},$ $i=1,2,3,$ be a Frenet frame field
on $\alpha$, $\alpha^{\ast}$ and we choose any point $\sigma_{0}\in I.$ If
$\mathbf{e}_{2}$ and $\mathbf{e}_{2}^{\ast}$ are timelike vectors, There
exists a Lorentzian motion $\varphi$ of $\mathbb{E}_{1}^{3}$ such that
\[
\varphi\left(  \alpha\left(  \sigma_{0}\right)  \right)  =\alpha^{\ast}\left(
\sigma_{0}\right)  \text{, \ \ }\varphi\left(  \mathbf{e}_{1}\left(
\sigma_{0}\right)  \right)  =\mathbf{e}_{1}^{\ast}\left(  \sigma_{0}\right)
\text{ \ \ and \ }\varphi\left(  \mathbf{e}_{i}\left(  \sigma_{0}\right)
\right)  =-\mathbf{e}_{i}^{\ast}\left(  \sigma_{0}\right)  \text{ for }i=2,3.
\]
Let's consider the function $\Psi:I\rightarrow%
\mathbb{R}
$ defined by
\[
\Psi\left(  \sigma\right)  =\left\Vert \varphi\left(  \mathbf{e}_{1}\left(
\sigma\right)  \right)  -\mathbf{e}_{1}^{\ast}\left(  \sigma\right)
\right\Vert ^{2}+\left\Vert \varphi\left(  \mathbf{e}_{2}\left(
\sigma\right)  \right)  +\mathbf{e}_{2}^{\ast}\left(  \sigma\right)
\right\Vert ^{2}+\left\Vert \varphi\left(  \mathbf{e}_{3}\left(
\sigma\right)  \right)  +\mathbf{e}_{3}^{\ast}\left(  \sigma\right)
\right\Vert ^{2}.
\]
Then
\[
\frac{d\Psi}{d\sigma}=2\left(  \tilde{\tau}+\tilde{\tau}^{\ast}\right)
\left(  \varphi\left(  \mathbf{e}_{3}\right)  \cdot\mathbf{e}_{2}%
+\varphi\left(  \mathbf{e}_{2}\right)  \cdot\mathbf{e}_{3}\right)  =0.
\]
Due to $\Psi\left(  \sigma_{0}\right)  =0,$ we can write
\[
\varphi\left(  \mathbf{e}_{1}\left(  \sigma_{0}\right)  \right)
=\mathbf{e}_{1}^{\ast}\left(  \sigma\right)  \text{ \ \ and \ }\varphi\left(
\mathbf{e}_{i}\left(  \sigma\right)  \right)  =-\mathbf{e}_{i}^{\ast}\left(
\sigma\right)  \text{ for }i=2,3\text{ \ \ \ }\forall\sigma\in I.
\]
\ \ \ \ The map $g=\mu\varphi:\mathbb{E}_{1}^{3}\rightarrow\mathbb{E}_{1}^{3}$
is a p-similarity of $\mathbb{E}_{1}^{3}$. We examine the function
$\Phi:I\rightarrow%
\mathbb{R}
$ such that
\[
\Phi\left(  \sigma\right)  =\left\Vert \frac{d}{d\sigma}g\left(  \alpha\left(
\sigma\right)  \right)  -\frac{d}{d\sigma}\alpha^{\ast}\left(  \sigma\right)
\right\Vert ^{2}\text{ \ \ \ for }\forall\sigma\in I.
\]
Since we have $\dfrac{d\Phi}{d\sigma}=0$ and $\Phi\left(  \sigma_{0}\right)
=0,$ we get $\Phi\left(  \sigma\right)  =0$ for any $\sigma\in I.$ Namely, we
can write $\dfrac{d}{d\sigma}g\left(  \alpha\left(  \sigma\right)  \right)
=\dfrac{d}{d\sigma}\alpha^{\ast}\left(  \sigma\right)  $ or equivalently
$\alpha^{\ast}\left(  \sigma\right)  =g\left(  \alpha\left(  \sigma\right)
\right)  +\mathbf{b}$ where $\mathbf{b}$ is a constant vector. So, we have
$f=\vartheta\circ g,$ where $\vartheta:\mathbb{E}_{1}^{3}\rightarrow
\mathbb{E}_{1}^{3}$ is a translation function determined by $\mathbf{b,}$ is
orientation-preserving p-similarity transformation such that the image of the
spacelike curve $\alpha$ under $f$ is the spacelike curve $\alpha^{\ast},$
i.e. $\alpha^{\ast}=f\circ\alpha.$

\qquad In the same way, if we take $\mathbf{e}_{3}$ and $\mathbf{e}_{3}^{\ast
}$ as timelike vectors, we can find an orientation-preserving p-similarity $f$
which provides $\alpha^{\ast}=f\circ\alpha$ such that the functions $\Psi$ and
$\Phi$ are respectively defined by
\begin{align*}
\Psi\left(  \sigma\right)   &  =\left\Vert \varphi\left(  \mathbf{e}%
_{1}\left(  \sigma\right)  \right)  -\mathbf{e}_{1}^{\ast}\left(
\sigma\right)  \right\Vert ^{2}+\left\Vert \varphi\left(  \mathbf{e}%
_{2}\left(  \sigma\right)  \right)  -\mathbf{e}_{2}^{\ast}\left(
\sigma\right)  \right\Vert ^{2}+\left\Vert \varphi\left(  \mathbf{e}%
_{3}\left(  \sigma\right)  \right)  -\mathbf{e}_{3}^{\ast}\left(
\sigma\right)  \right\Vert ^{2},\\
\Phi\left(  \sigma\right)   &  =\left\Vert \frac{d}{d\sigma}g\left(
\alpha\left(  \sigma\right)  \right)  -\frac{d}{d\sigma}\alpha^{\ast}\left(
\sigma\right)  \right\Vert ^{2}\text{ \ \ \ for }\forall\sigma\in I.
\end{align*}

\end{proof}

\section{Construction of the non-lightlike Frenet curves by curves on the
Lorentzian and hyperbolic unit sphere}

\qquad Let $\mathbf{c}:I\rightarrow S_{1}^{2}$ be non-lightlike spherical
curve with the arc length parameter $\sigma$. The orthonormal frame $\left\{
\mathbf{c}\left(  \sigma\right)  ,\mathbf{t}\left(  \sigma\right)
,\mathbf{q}\left(  \sigma\right)  \right\}  $ along $\mathbf{c}$ is called
the\emph{ Sabban frame }of $\mathbf{c}$ if $\mathbf{t}\left(  \sigma\right)
=\dfrac{d\mathbf{c}}{d\sigma}$ is the unit tangent vector of $\mathbf{c}$ and
$\mathbf{q}\left(  \sigma\right)  =\mathbf{c}\left(  \sigma\right)
\times\mathbf{t}\left(  \sigma\right)  .$ Then we state spherical
Frenet-Serret formulas of the non-lightlike curve $\mathbf{c}$.

\qquad If the curve $\mathbf{c}$ is a timelike curve, i.e. $\mathbf{t}\left(
\sigma\right)  $ is timelike vector, we have the following spherical
Frenet-Serret formulas of $\mathbf{c}$:%
\begin{equation}
\frac{d}{d\sigma}%
\begin{bmatrix}
\mathbf{c}\\
\mathbf{t}\\
\mathbf{q}%
\end{bmatrix}
=%
\begin{bmatrix}
0 & 1 & 0\\
1 & 0 & k_{g}\\
0 & k_{g} & 0
\end{bmatrix}%
\begin{bmatrix}
\mathbf{c}\\
\mathbf{t}\\
\mathbf{q}%
\end{bmatrix}
\label{11}%
\end{equation}
\ \ \ \ If $\mathbf{q}\left(  \sigma\right)  $ is a timelike vector, we have
the following spherical Frenet-Serret formulas of $\mathbf{c}$:%
\begin{equation}
\frac{d}{d\sigma}%
\begin{bmatrix}
\mathbf{c}\\
\mathbf{t}\\
\mathbf{q}%
\end{bmatrix}
=%
\begin{bmatrix}
0 & -1 & 0\\
1 & 0 & k_{g}\\
0 & k_{g} & 0
\end{bmatrix}%
\begin{bmatrix}
\mathbf{c}\\
\mathbf{t}\\
\mathbf{q}%
\end{bmatrix}
\end{equation}
If $\mathbf{c}:I\rightarrow H_{0}^{2}$ is a spacelike spherical curve with the
arc length parameter $\sigma$, then spherical Frenet-Serret formulas of
$\mathbf{c}$ are
\begin{equation}
\frac{d}{d\sigma}%
\begin{bmatrix}
\mathbf{c}\\
\mathbf{t}\\
\mathbf{q}%
\end{bmatrix}
=%
\begin{bmatrix}
0 & -1 & 0\\
-1 & 0 & k_{g}\\
0 & -k_{g} & 0
\end{bmatrix}%
\begin{bmatrix}
\mathbf{c}\\
\mathbf{t}\\
\mathbf{q}%
\end{bmatrix}
\label{12}%
\end{equation}
since $\mathbf{c}\left(  \sigma\right)  $ is a timelike vector. $k_{g}\left(
\sigma\right)  =\varepsilon_{\mathbf{q}}\det\left(  \mathbf{c}\left(
\sigma\right)  \mathbf{,t}\left(  \sigma\right)  \mathbf{,}\dfrac{d\mathbf{t}%
}{d\sigma}\left(  \sigma\right)  \right)  $ is the geodesic curvature of
$\mathbf{c}$ for three different spherical Frenet-Serret formulas\textbf{.}

\qquad Let $k:I\rightarrow%
\mathbb{R}
$ be a function of class $C^{1}.$ We can describe a non-lightlike curve
$\alpha:I\rightarrow\mathbb{E}_{1}^{3}$ given by
\begin{equation}
\alpha\left(  \sigma\right)  =b\int e^{\int k\left(  \sigma\right)  d\sigma
}\mathbf{c}\left(  \sigma\right)  d\sigma+\mathbf{a}, \label{13}%
\end{equation}
where $\mathbf{a}$ is a constant vector and $b$ is a real constant. The fact
that $\sigma$ is arc spherical length parameter of $\alpha$ can be easily seen
because we have $\frac{\frac{d\alpha}{d\sigma}}{\left\Vert \frac{d\alpha
}{d\sigma}\right\Vert }=\mathbf{c}\left(  \sigma\right)  $. Then, we can state
a description of all Frenet curves in Minkowski 3-space.

\begin{proposition}
\label{prop}The non-lightlike curve $\alpha$ defined by $\left(
\ref{13}\right)  $ is a Frenet curve with shape curvature $\tilde{\kappa
}=k\left(  \sigma\right)  $ and shape torsion $\tilde{\tau}=\varepsilon
_{\mathbf{q}}k_{g}\left(  \sigma\right)  $ in the Minkowski 3-space.
Furthermore, all non-lightlike Frenet curves can be obtained in this way.
\end{proposition}

\begin{proof}
First, from $\left(  \ref{13}\right)  $ we can write
\begin{align*}
\frac{d\alpha}{d\sigma} &  =be^{\int k\left(  \sigma\right)  d\sigma
}\mathbf{c}\left(  \sigma\right)  ,\text{ \ \ }\frac{d^{2}\alpha}{d\sigma^{2}%
}=be^{\int k\left(  \sigma\right)  d\sigma}\left[  k\left(  \sigma\right)
\mathbf{c}\left(  \sigma\right)  +\frac{d\mathbf{c}}{d\sigma}\right]  \\
\frac{d^{3}\alpha}{d\sigma^{3}} &  =be^{\int k\left(  \sigma\right)  d\sigma
}\left[  \left\{  k^{2}\left(  \sigma\right)  +\frac{dk}{d\sigma}\right\}
\mathbf{c}\left(  \sigma\right)  +2k\left(  \sigma\right)  \frac{d\mathbf{c}%
}{d\sigma}+\frac{d^{2}\mathbf{c}}{d\sigma^{2}}\right]  .
\end{align*}
Then, because of the equation
\[
\frac{d\alpha}{d\sigma}\times\frac{d^{2}\alpha}{d\sigma^{2}}=b^{2}e^{2\int
k\left(  \sigma\right)  d\sigma}\left(  \mathbf{c}\left(  \sigma\right)
\times\frac{d\mathbf{c}}{d\sigma}\right)  \neq0,
\]
we have $\alpha$ is non-lightlike Frenet curve. Using $\left(  \ref{8}\right)
$ and $\left(  \ref{9}\right)  $ we find that
\[
\tilde{\kappa}=k\left(  \sigma\right)  \ \ \ \text{and}\ \ \ \tilde{\tau}%
=\det\left(  \mathbf{c},\frac{d\mathbf{c}}{d\sigma},\dfrac{d\mathbf{t}%
}{d\sigma}\right)  =\varepsilon_{\mathbf{q}}k_{g}\left(  \sigma\right)  .
\]
\ \ \ \ Conversely, suppose that $\alpha:I\rightarrow\mathbb{E}_{1}^{3}$ is a
non-lightlike regular curve parameterized by a spherical arc length parameter
$\sigma.$ Denote by $\kappa\left(  \sigma\right)  $ and $\tau\left(
\sigma\right)  $ the curvature and the torsion of $\mathbf{c,}$ respectively.
Let $\mathbf{c}$ be the spherical indicator of $\alpha$ such that
$\mathbf{c}:I\rightarrow\mathbb{E}_{1}^{3}$ is given by
\begin{equation}
\mathbf{c}\left(  \sigma\right)  =\mathbf{e}_{1}\left(  \sigma\right)
=\frac{\frac{d\alpha}{d\sigma}}{\left\Vert \frac{d\alpha}{d\sigma}\right\Vert
}=\kappa\left(  \sigma\right)  \frac{d\alpha}{d\sigma}.\label{15}%
\end{equation}
We can say that $\sigma$ is an arc length parameter of $\mathbf{c}$ and
$k_{g}=\varepsilon_{\mathbf{q}}\det\left(  \mathbf{c}\left(  \sigma\right)
,\mathbf{t}\left(  \sigma\right)  ,\dfrac{d\mathbf{t}\left(  \sigma\right)
}{d\sigma}\right)  =\varepsilon_{\mathbf{q}}\tilde{\tau}$ is the geodesic
curvature of $\mathbf{c}$. If we take $k\left(  \sigma\right)  =\tilde{\kappa
}\left(  \sigma\right)  ,$ then
\begin{align*}
\int e^{\int k\left(  \sigma\right)  d\sigma}\mathbf{c}\left(  \sigma\right)
d\sigma &  =\int e^{\int-\frac{d\kappa}{\kappa d\sigma}d\sigma}\mathbf{c}%
\left(  \sigma\right)  d\sigma=e^{b_{0}}\int\frac{1}{\kappa}\mathbf{c}\left(
\sigma\right)  d\sigma\\
&  =e^{b_{0}}\int\frac{d\alpha}{d\sigma}d\sigma=e^{b_{0}}\alpha\left(
\sigma\right)  +\mathbf{a}_{0}%
\end{align*}
where $b_{0}$ is a real constant and $\mathbf{a}_{0}$ is a constant vector.
Hence, we can write
\[
\alpha\left(  \sigma\right)  =b\int e^{\int k\left(  \sigma\right)  d\sigma
}\mathbf{c}\left(  \sigma\right)  d\sigma+\mathbf{a.}%
\]

\end{proof}

\begin{theorem}
\label{var}(Existence Theorem) Let $z_{i}:I\rightarrow%
\mathbb{R}
,$ $i=1,2$, be two functions of class $C^{1}$ and $\mathbf{e}_{1}^{0},$
$\mathbf{e}_{2}^{0},$ $\mathbf{e}_{3}^{0}$ be an right-handed orthonormal
triad of vectors at a point $x_{0}$ in the Minkowski 3-space $\mathbb{E}%
_{1}^{3}.$ According to a p-similarity with center $x_{0}$ there exists a
unique non-lightlike curve $\alpha:I\rightarrow\mathbb{E}_{1}^{3}$ such that
$\alpha$ satisfies the following conditions:

$\left(  i\right)  $ There exists a $\sigma_{0}\in I$ such that $\alpha\left(
\sigma_{0}\right)  =x_{0}$ and the Frenet-Serret frame of $\alpha$ at $x_{0}$
is $\left\{  \mathbf{e}_{1}^{0},\mathbf{e}_{2}^{0},\mathbf{e}_{3}^{0}\right\}
.$

$\left(  ii\right)  $ $\tilde{\kappa}\left(  \sigma\right)  =z_{1}\left(
\sigma\right)  $ and $\tilde{\tau}\left(  \sigma\right)  =\varepsilon
_{\mathbf{e}_{3}^{0}}z_{2}\left(  \sigma\right)  $ for any $\sigma\in I.$
\end{theorem}

\begin{proof}
We consider the system of differential equations
\begin{equation}
\frac{d\mathbf{X}}{d\sigma}\left(  \sigma\right)  =\mathbf{M}\left(
\sigma\right)  \mathbf{X}\left(  \sigma\right)  \label{14}%
\end{equation}
where $\mathbf{X}\left(  \sigma\right)  =%
\begin{bmatrix}
\mathbf{c}\left(  \sigma\right)   & \mathbf{t}\left(  \sigma\right)   &
\mathbf{q}\left(  \sigma\right)
\end{bmatrix}
$ and $\mathbf{M}$ is one of the following matrices depending on whether
$\mathbf{t}\left(  \sigma\right)  \mathbf{,}$ $\mathbf{c}\left(
\sigma\right)  $ or $\mathbf{q}\left(  \sigma\right)  $ is a timelike vector,
respectively:
\[%
\begin{bmatrix}
0 & 1 & 0\\
1 & 0 & z_{2}\\
0 & z_{2} & 0
\end{bmatrix}
\text{,\ \ }%
\begin{bmatrix}
0 & -1 & 0\\
-1 & 0 & z_{2}\\
0 & -z_{2} & 0
\end{bmatrix}
\text{ \ or\ \ }%
\begin{bmatrix}
0 & -1 & 0\\
1 & 0 & z_{2}\\
0 & z_{2} & 0
\end{bmatrix}
.
\]
The system $\left(  \ref{14}\right)  $ has an unique solution $\mathbf{X}%
\left(  \sigma\right)  $ which satisfies initial conditions $\mathbf{X}\left(
\sigma_{0}\right)  =%
\begin{bmatrix}
\mathbf{e}_{1}^{0} & \mathbf{e}_{2}^{0} & \mathbf{e}_{3}^{0}%
\end{bmatrix}
$ for $\sigma_{0}\in I.$ If $\mathbf{I}$ is the unit matrix and$\ \mathbf{X}%
^{t}$ is the transposed matrix of $\mathbf{X}\left(  \sigma\right)  ,$ then we
can obtain
\begin{align*}
\frac{d}{d\sigma}\left(  \mathbf{I}^{\ast}\mathbf{X}^{t}\mathbf{I}^{\ast
}\mathbf{X}\right)   &  =\mathbf{I}^{\ast}\frac{d}{d\sigma}\mathbf{X}%
^{t}\mathbf{I}^{\ast}\mathbf{X}+\mathbf{I}^{\ast}\mathbf{X}^{t}\mathbf{I}%
^{\ast}\frac{d}{d\sigma}\mathbf{X}\\
&  =\mathbf{I}^{\ast}\mathbf{X}^{t}\mathbf{M}^{t}\mathbf{I}^{\ast}%
\mathbf{X}+\mathbf{I}^{\ast}\mathbf{X}^{t}\mathbf{I}^{\ast}\mathbf{MX}\\
&  =\mathbf{I}^{\ast}\mathbf{X}^{t}\left(  \mathbf{M}^{t}\mathbf{I}^{\ast
}+\mathbf{I}^{\ast}\mathbf{M}\right)  \mathbf{X}=0
\end{align*}
using the equation $\mathbf{M}^{t}\mathbf{I}^{\ast}+\mathbf{I}^{\ast
}\mathbf{M}=\left[  0\right]  _{3\times3},$ where $\mathbf{I}^{\ast
}=diag\left(  -1,1,1\right)  ,$ $diag\left(  1,-1,1\right)  $ or $diag\left(
1,1,-1\right)  ,$ when $\mathbf{c}\left(  \sigma\right)  $, $\mathbf{t}\left(
\sigma\right)  $\textbf{ }or $\mathbf{q}\left(  \sigma\right)  $ is a timelike
vector, respectively. Also, we have $\mathbf{I}^{\ast}\mathbf{X}^{t}\left(
\sigma_{0}\right)  \mathbf{I}^{\ast}\mathbf{X}\left(  \sigma_{0}\right)
=\mathbf{I}$ since $\left\{  \mathbf{e}_{1}^{0},\mathbf{e}_{2}^{0}%
,\mathbf{e}_{3}^{0}\right\}  $ is the orthonormal frame. As a result, we find
$\mathbf{I}^{\ast}\mathbf{X}^{t}\left(  \sigma\right)  \mathbf{I}^{\ast
}\mathbf{X}\left(  \sigma\right)  =\mathbf{I}$ for any $\sigma\in I.$ This
means that the vector fields $\mathbf{t}\left(  \sigma\right)  \mathbf{,}$
$\mathbf{c}\left(  \sigma\right)  $ and $\mathbf{q}\left(  \sigma\right)  $
form a right-handed orthonormal frame field.

\qquad Let $\alpha:I\rightarrow\mathbb{E}_{1}^{3}$ be the regular
non-lightlike curve given by
\[
\alpha\left(  \sigma\right)  =b\int_{\sigma_{0}}^{\sigma}e^{\int z_{1}\left(
\sigma\right)  d\sigma}\mathbf{c}\left(  \sigma\right)  d\sigma+x_{0},\text{
\ \ \ \ \ \ \ \ }\sigma\in I,\text{ }b>0.
\]
By the proposition $\left(  \ref{prop}\right)  ,$ we get that the
Frenet-Serret frame field of $\alpha$ is
\[
\left\{  \mathbf{e}_{1}\left(  \sigma\right)  =\mathbf{c}\left(
\sigma\right)  \mathbf{,}\text{ }\mathbf{\mathbf{e}_{2}}\left(  \sigma\right)
\mathbf{=\mathbf{t}\left(  \sigma\right)  },\text{ }\mathbf{e}_{3}\left(
\sigma\right)  =\mathbf{q}\left(  \sigma\right)  \right\}
\]
and Frenet-Serret frame of $\alpha$ at $x_{0}=\alpha\left(  \sigma_{0}\right)
$ is
\[
\left\{  \mathbf{e}_{1}^{0}\left(  \sigma_{0}\right)  =\mathbf{c}\left(
\sigma_{0}\right)  \mathbf{,}\text{ }\mathbf{\mathbf{e}}_{2}^{0}\left(
\sigma_{0}\right)  \mathbf{=t}\left(  \sigma_{0}\right)  ,\text{ }%
\mathbf{e}_{3}^{0}\left(  \sigma_{0}\right)  =\mathbf{q}\left(  \sigma
_{0}\right)  \right\}  .
\]
Besides, the functions $z_{1}$ and $\varepsilon_{\mathbf{e}_{3}^{0}}z_{2}$ are
the p-shape curvature and p-shape torsion of $\alpha$, respectively.
\end{proof}

\qquad From Theorems $\ref{teklik}$ and $\ref{var},$ we get the following
theorem which is an analogue of the fundamental theorem of curves.

\begin{theorem}
\label{tek}Let $z_{i}:I\rightarrow%
\mathbb{R}
,$ $i=1,2$, be two functions of class $C^{1}.$ According to p-similarity there
exists a unique non-lightlike Frenet curve with p-shape curvature $z_{1}$ and
p-shape torsion $z_{2}.$
\end{theorem}

\subsection{Forming a non-lightlike curve from its p-shape}

\qquad Let $\alpha:I\rightarrow\mathbb{E}_{1}^{3}$ be a non-lightlike curve
with the spherical arc length parameter $\sigma$ such that the ordered pair
$\left(  \tilde{\kappa}_{\alpha},\tilde{\tau}_{\alpha}\right)  $ is p-shape of
the $\alpha$ defined by $\left(  \ref{000}\right)  .$ From the Theorem
$\ref{tek}$ we have that $\alpha$ is uniquely determined by its p-shape
according to p-similarity in the Minkowski 3-space. First we define fixed
right-handed orthonormal triad of non-lightlike vectors $\mathbf{e}_{1}%
^{0}\mathbf{,}$ $\mathbf{\mathbf{e}}_{2}^{0},$ $\mathbf{e}_{3}^{0}.$ When
$\mathbf{t}\left(  \sigma\right)  \mathbf{,}$ $\mathbf{c}\left(
\sigma\right)  $ or $\mathbf{q}\left(  \sigma\right)  $ is timelike vector, we
take respectively differential equations
\begin{equation}
\frac{d\mathbf{c}}{d\sigma}=\mathbf{t}\left(  \sigma\right)  ,\text{
\ \ }\frac{d\mathbf{t}}{d\sigma}=\mathbf{c}\left(  \sigma\right)  +\tilde
{\tau}_{\alpha}\mathbf{q}\left(  \sigma\right)  ,\text{ \ \ \ }\frac
{d\mathbf{q}}{d\sigma}=\tilde{\tau}_{\alpha}\mathbf{t}\left(  \sigma\right)
\label{d1}%
\end{equation}%
\begin{equation}
\frac{d\mathbf{c}}{d\sigma}=-\mathbf{t}\left(  \sigma\right)  ,\text{
\ \ }\frac{d\mathbf{t}}{d\sigma}=-\mathbf{c}\left(  \sigma\right)
+\tilde{\tau}_{\alpha}\mathbf{q}\left(  \sigma\right)  ,\text{ \ \ \ }%
\frac{d\mathbf{q}}{d\sigma}=-\tilde{\tau}_{\alpha}\mathbf{t}\left(
\sigma\right)  \label{d2}%
\end{equation}%
\begin{equation}
\frac{d\mathbf{c}}{d\sigma}=-\mathbf{t}\left(  \sigma\right)  ,\text{
\ \ }\frac{d\mathbf{t}}{d\sigma}=\mathbf{c}\left(  \sigma\right)  -\tilde
{\tau}_{\alpha}\mathbf{q}\left(  \sigma\right)  ,\text{ \ \ \ }\frac
{d\mathbf{q}}{d\sigma}=-\tilde{\tau}_{\alpha}\mathbf{t}\left(  \sigma\right)
.\label{d3}%
\end{equation}
The unique solution of one of these differential equations with initial
conditions $\mathbf{e}_{1}^{0}\mathbf{,}$ $\mathbf{\mathbf{e}}_{2}^{0},$
$\mathbf{e}_{3}^{0},$ determine a spherical non-lightlike curve $\mathbf{c=c}%
\left(  \sigma\right)  $ such that $\mathbf{c}\left(  \sigma_{0}\right)
=\mathbf{e}_{1}^{0}$ for some $\sigma_{0}\in I.$ Let $\rho\left(
\sigma\right)  =\int_{\sigma_{1}}^{\sigma}\tilde{\kappa}_{\alpha}\left(
\sigma\right)  d\sigma$ for fixed $\sigma_{1}\in I$. Using the equation
$\left(  \ref{13}\right)  $ and proposition $\ref{prop}$ we can find the
non-lightlike curve
\begin{equation}
\alpha\left(  \sigma\right)  =\alpha_{0}+\int_{\sigma_{0}}^{\sigma}%
e^{\rho\left(  \sigma\right)  }\mathbf{c}\left(  \sigma\right)  d\sigma
\label{d0}%
\end{equation}
passes through a point $\alpha_{0}=\alpha\left(  \sigma_{0}\right)  .$ Now, we
show a few examples of the non-lightlike curves constructed by above procedure.

\begin{example}
\label{or}Let p-shape $\left(  \tilde{\kappa}_{\alpha},\tilde{\tau}_{\alpha
}\right)  $ of the $\alpha:I\rightarrow\mathbb{E}_{1}^{3}$ be $\left(
0,a\right)  ,$ where $a\neq0$ is real constant. We can find $\rho\left(
\sigma\right)  =0$ for any $\sigma\in I$.

$\mathbf{i)}$ We take the unit vector $\mathbf{t}\left(  \sigma\right)  $ as
timelike vector. Choose initial conditions
\begin{equation}
\mathbf{e}_{1}^{0}=\left(  0,-\frac{1}{\sqrt{1+a^{2}}},\frac{a}{\sqrt{1+a^{2}%
}}\right)  \mathbf{,}\text{ }\mathbf{\mathbf{e}}_{2}^{0}=\left(  1,0,0\right)
,\text{ }\mathbf{e}_{3}^{0}=\left(  0,\frac{a}{\sqrt{1+a^{2}}},\frac{1}%
{\sqrt{1+a^{2}}}\right)  . \label{i1}%
\end{equation}
Then, the system $\left(  \ref{d1}\right)  $ describes a spherical timelike
curve $\mathbf{c}:I\rightarrow S_{1}^{2}$ defined by%
\begin{equation}
\mathbf{c}\left(  \sigma\right)  =\left(  \frac{1}{\sqrt{1+a^{2}}}\sinh\left(
\sqrt{1+a^{2}}\sigma\right)  ,-\frac{1}{\sqrt{1+a^{2}}}\cosh\left(
\sqrt{1+a^{2}}\sigma\right)  ,\frac{a}{\sqrt{1+a^{2}}}\right)  \label{i2}%
\end{equation}
with $\mathbf{c}\left(  0\right)  =\mathbf{e}_{1}^{0},$ in the Minkowski
3-space. Solving the equation $\left(  \ref{d0}\right)  $ we obtain the
spacelike curve parameterized by
\[
\alpha\left(  \sigma\right)  =\left(  \frac{1}{1+a^{2}}\cosh\left(
\sqrt{1+a^{2}}\sigma\right)  ,-\frac{1}{1+a^{2}}\sinh\left(  \sqrt{1+a^{2}%
}\sigma\right)  ,\frac{a}{\sqrt{1+a^{2}}}\sigma\right)  ,\ \ \sigma\in I.
\]

$\mathbf{ii)}$ Let the unit vector $\mathbf{c}\left(  \sigma\right)  $ be
timelike vector. We choose another initial conditions
\[
\mathbf{e}_{1}^{0}=\left(  \frac{a}{\sqrt{a^{2}-1}},0,\frac{1}{\sqrt{a^{2}-1}%
}\right)  \mathbf{,}\ \mathbf{\mathbf{e}}_{2}^{0}=\left(  0,1,0\right)
,\ \mathbf{e}_{3}^{0}=\left(  \frac{1}{\sqrt{a^{2}-1}},0,\frac{a}{\sqrt
{a^{2}-1}}\right)
\]
where $a^{2}>1.$ Then, the system $\left(  \ref{d2}\right)  $ describes a
spherical spacelike curve $\mathbf{c}:I\rightarrow H_{0}^{2}$ defined by%
\begin{equation}
\mathbf{c}\left(  \sigma\right)  =\left(  \frac{a}{\sqrt{a^{2}-1}},\frac
{1}{\sqrt{a^{2}-1}}\sin\left(  \sqrt{a^{2}-1}\sigma\right)  ,\frac{1}%
{\sqrt{a^{2}-1}}\cos\left(  \sqrt{a^{2}-1}\sigma\right)  \right)  \label{i3}%
\end{equation}
with $\mathbf{c}\left(  0\right)  =\mathbf{e}_{1}^{0},$ in the Minkowski
3-space. Solving the equation $\left(  \ref{d0}\right)  $ we obtain the
timelike curve given by%
\[
\alpha\left(  \sigma\right)  =\left(  \frac{a}{\sqrt{a^{2}-1}}\sigma,-\frac
{1}{a^{2}-1}\cos\left(  \sqrt{a^{2}-1}\sigma\right)  ,\frac{1}{a^{2}-1}%
\sin\left(  \sqrt{a^{2}-1}\sigma\right)  \right)  .
\]

$\mathbf{iii)}$ Let the unit vector $\mathbf{q}\left(  \sigma\right)  $ be
timelike vector. Choose another initial conditions
\[
\mathbf{e}_{1}^{0}=\left(  \frac{1}{\sqrt{a^{2}-1}},0,\frac{a}{\sqrt{a^{2}-1}%
}\right)  \mathbf{,}\ \mathbf{\mathbf{e}}_{2}^{0}=\left(  0,1,0\right)
,\ \mathbf{e}_{3}^{0}=\left(  \frac{a}{\sqrt{a^{2}-1}},0,\frac{1}{\sqrt
{a^{2}-1}}\right)
\]
where $a^{2}>1.$ Then, the system $\left(  \ref{d3}\right)  $ describes a
spherical spacelike curve $\mathbf{c}:I\rightarrow S_{1}^{2}$ defined by%
\begin{equation}
\mathbf{c}\left(  \sigma\right)  =\left(  \frac{1}{\sqrt{a^{2}-1}}\cosh\left(
\sqrt{a^{2}-1}\sigma\right)  ,\frac{1}{\sqrt{a^{2}-1}}\sinh\left(  \sqrt
{a^{2}-1}\sigma\right)  ,\frac{a}{\sqrt{a^{2}-1}}\right)  \label{i4}%
\end{equation}
with $\mathbf{c}\left(  0\right)  =\mathbf{e}_{1}^{0},$ in the Minkowski
3-space. Solving the equation $\left(  \ref{d0}\right)  $ we obtain the
spacelike Frenet curve given by%
\[
\alpha\left(  \sigma\right)  =\left(  \frac{1}{a^{2}-1}\sinh\left(
\sqrt{a^{2}-1}\sigma\right)  ,\frac{1}{a^{2}-1}\cosh\left(  \sqrt{a^{2}%
-1}\sigma\right)  ,\frac{a}{\sqrt{a^{2}-1}}\sigma\right)  .
\]

\end{example}

\begin{example}
Let $\alpha:I\rightarrow\mathbb{E}_{1}^{3}$ be a non-lightlike curve with
p-shape $\left(  \tilde{\kappa}_{\alpha},\tilde{\tau}_{\alpha}\right)
=\left(  1/\sigma,a\right)  $ where $a\neq0$ is real constants. Because of
$\rho\left(  \sigma\right)  =\ln\sigma,$ the parametric equation of the
non-lightlike curve $\alpha$ is given by
\[
\alpha\left(  \sigma\right)  =\left(  \frac{t\cosh t-\sinh t}{\left(
1+a^{2}\right)  ^{3/2}},\frac{\cosh t-t\sinh t}{\left(  1+a^{2}\right)
^{3/2}},\frac{at^{2}}{2\left(  1+a^{2}\right)  ^{3/2}}\right)
\]
where $t=\sqrt{1+a^{2}}\sigma.$ As in the Example $\ref{or}$ we take the same
spherical timelike curve $\mathbf{c}=\mathbf{c}\left(  \sigma\right)  $
parameterized by $\left(  \ref{i2}\right)  .$
\end{example}

\qquad Now, we study non-lightlike self-similar curves in $\mathbb{E}_{1}^{3}%
$. A non-lightlike curve $\alpha:I\rightarrow\mathbb{E}_{1}^{3}$ is called
\emph{self-similar }if any p-similarity $f\in G$ conserve globally $\alpha$
and $G$ acts transitively on $\alpha$ where $G$ is a one-parameter subgroup of
\textbf{Sim}$\left(  \mathbb{E}_{1}^{3}\right)  .$ This means that p-shape
curvatures are constant. In fact, let $p_{1}=\alpha\left(  s_{1}\right)  $ and
$p_{2}=\alpha\left(  s_{2}\right)  $ be two different points lying on
$\alpha.$ Since $G$ acts transitively on $\alpha,$ there is a similarity $f\in
G$ such that $f\left(  p_{1}\right)  =p_{2}.$ Then, we find $\tilde{\kappa
}_{\alpha}\left(  s_{1}\right)  =\tilde{\kappa}_{\alpha}\left(  s_{2}\right)
$ and $\tilde{\tau}_{\alpha}\left(  s_{1}\right)  =\tilde{\tau}_{\alpha
}\left(  s_{2}\right)  $ because of the invariance of p-shape curvatures.

\qquad Now, we shall state the parametrization of all non-lightlike
self-similar curves. Let $\alpha:I\rightarrow\mathbb{E}_{1}^{3}$ be a
non-lightlike curve with the p-shape $\left(  \tilde{\kappa}_{\alpha}%
,\tilde{\tau}_{\alpha}\right)  =\left(  b,a\right)  $ where $a\neq0$ and
$b\neq0$ are real constants. Firstly we take $\mathbf{t}\left(  \sigma\right)
$ as a timelike unit vector. Choosing initial conditions $\left(
\ref{i1}\right)  $ as in the Example $\ref{or},$ we get the same spherical
timelike curve $\left(  \ref{i2}\right)  $ which is a pseudo-circle with a
radius $1/\sqrt{1+a^{2}}$. Also, we have $\rho\left(  \sigma\right)  =\int
_{0}^{\sigma}bd\sigma=b\sigma$ for $\sigma\in I$. Solving the equation
$\left(  \ref{d0}\right)  ,$ we obtain a spacelike self-similar curve which
has the spherical arc length parametrization as the following
\[
\alpha_{\mathbf{t}}\left(  \sigma\right)  =\left(  \frac{e^{b\sigma}}{\left(
b^{2}-q^{2}\right)  }\left(  \frac{b}{q}\sinh(q\sigma)-\cosh(q\sigma)\right)
,\frac{e^{b\sigma}}{\left(  b^{2}-q^{2}\right)  }\left(  \sinh(q\sigma
)-\frac{b}{q}\cosh(q\sigma)\right)  ,\frac{a}{bq}e^{b\sigma}\right)
\]
where $q=\sqrt{1+a^{2}}.$

\qquad If we take $\mathbf{c}\left(  \sigma\right)  $ as a timelike unit
vector, by using $\left(  \ref{i3}\right)  $ we obtain similarly a timelike
self-similar curve given by
\[
\alpha_{\mathbf{c}}\left(  \sigma\right)  =\left(  \frac{a}{bn}e^{b\sigma
},\frac{e^{b\sigma}}{\left(  b^{2}-n^{2}\right)  }\left(  \frac{b}{n}%
\sin(n\sigma)-\cos(n\sigma)\right)  ,\frac{e^{b\sigma}}{\left(  b^{2}%
-n^{2}\right)  }\left(  \frac{b}{n}\cos(n\sigma)+\sin(n\sigma)\right)
\right)
\]
where $n=\sqrt{a^{2}-1}.$

\qquad If we take $\mathbf{q}\left(  \sigma\right)  $ as a timelike unit
vector, by using $\left(  \ref{i4}\right)  $ we get a spacelike self-similar
curve as the following
\[
\alpha_{\mathbf{q}}\left(  \sigma\right)  =\left(  \frac{e^{b\sigma}}{\left(
b^{2}-n^{2}\right)  }\left(  \frac{b}{n}\cosh(n\sigma)-\sinh(n\sigma)\right)
,\frac{e^{b\sigma}}{\left(  b^{2}-n^{2}\right)  }\left(  \frac{b}{n}%
\sinh(n\sigma)-\cosh(n\sigma)\right)  ,\frac{a}{bn}e^{b\sigma}\right)  .
\]

\bigskip

Hakan \c{S}im\c{s}ek and Mustafa \"{O}zdemir

Department of Mathematics

Akdeniz University

Antalya, TURKEY;

e-mail: hakansimsek@akdeniz.edu.tr.

\ \ \ \ \ \ \ \ \ \ \ mozdemir@akdeniz.edu.tr


\begin{thebibliography}{99}                                                                                               %


\bibitem {vision00}A. Brook, A. M. Bruckstein and Ron Kimmel, \textit{On
Similarity-Invariant Fairness Measures, }LNCS 3459, pp. 456--467, 2005.

\bibitem {fractal}B. B. Mandelbrot, The Fractal Geometry of Nature, New York:
W. H. Freeman, 1983.

\bibitem {semi riemann}B. O'Neill, \textit{Semi-Riemannian Geometry with
Applications to Relativity}. Academic Press Inc., London, 1983.

\bibitem {normal}D. A. Singer\textsc{ }and\textsc{\textsc{ }}D.
H.\textsc{\textsc{ }}Steinberg, \textit{Normal Forms in Lorentzian Spaces.
}Nova J. Algebra Geo., 1994, Vol3 (1), pp. 1-9.

\bibitem {vision01}D. Xu and H. Li, \textit{3-D Curve Moment Invariants for
Curve Recognition, }Lecture Notes in Control and Information Sciences, 345,
pp. 572-577, 2006.

\bibitem {vision0}H. Sahbi, \textit{Kernel PCA for similarity invariant shape
recognition, }Neurocomputing, 70 (2007), 3034--3045.

\bibitem {fractal1}J. E. Hutchinson, \textit{Fractals and Self-Similarity},
Indiana University Mathematics Journal, Vol. 30, N:5, 1981.

\bibitem {vision1}J. G. Alc\'{a}zar, C. Hermosoa and G. Muntinghb,
\textit{Detecting similarity of rational plane curves, }Journal of
Computational and Applied Mathematics, 269 (2014), 1--13.

\bibitem {biharmonic}J. Inoguchi, \textit{Biharmonic curves in Minkowski
3-space}, International Journal of Mathematics and Mathematical Sciences 21
(2003): 1365-1368.

\bibitem {chaos0}KS. Chou and C. Qu\textsc{, }\textit{Integrable equations
arising from motions of plane curves, }Pysica D, 162 (2002), 9-33.

\bibitem {chaos}KS. Chou and C. Qu\textsc{, }\textit{Motions of curves in
similarity geometries and Burgers-mKdV hierarchies, Chaos, Solitons }$\And$
Fractals 19 (2004), 47-53.

\bibitem {fractal2}K. Falconer, \textit{Fractal Geometry: Mathematical
Foundations and Applications}, Second Edition, John Wiley \& Sons, Ltd. , 2003.

\bibitem {fractal3}M. F. Barnsley, J. E. Hutchinson and \"{O}. Stenflo,
\textit{V-variable fractals: Fractals with partial self similarity, }Advances
in Mathematics, 218, pp. 2051-2088, 2008.\textit{ }

\bibitem {fractal30}M. F. Barnsley and S. Demko, \textit{Iterated function
systems and the global construction of fractals}, Proc. R. Soc. London Ser A
399: 243--275, 1985.

\bibitem {berger}M. Berger, \textit{Geometry I}. Springer, New York 1998.

\bibitem {gurses}M. G\"{u}rses\textsc{, }\textit{Motion of curves on
two-dimensional surfaces and soliton equations, }Physics Letters A, 241
(1998), 329-334.

\bibitem {rotation}M. \"{O}zdemir, A. A. Ergin, \textit{Rotations with unit
timelike quaternions in Minkowski 3-space}, Journal of Geometry and Physics
56, 322--336, 2006.

\bibitem {minkowski}M. \"{O}zdemir, A. A. Ergin\textsc{, }\textit{Spacelike
Darboux Curves in Minkowski 3-Space}, Differ. Geom. Dyn. Syst. 9, 131-137, 2007.

\bibitem {inoguchi}Q. Ding and J. Inoguchi, \textit{Schr\"{o}dinger flows,
binormal motion for curves and the second AKNS-hierarchies}, Chaos, Solitons
and Fractals, 21, 669--677, 2004.

\bibitem {shape}R.\textsc{ }Encheva and\textsc{ }G. Georgiev, \textit{Shapes
of space curves}. J. Geom. Graph. 7 (2003), 145-155.

\bibitem {vision2}S. Z. Li, \textit{Invariant Representation, Matching and
Pose Estimation of 3D Space Curves Under Similarity Transformation, }Pattern
Recognition, Vol. 30, No. 3, pp. 447458, 1997.

\bibitem {vision3}S. Z. Li, \textit{Similarity Invariants for 3D Space Curve
Matching}, In Proceedings of the First Asian Conference on Computer Vision,
pp. 454-457, Japan 1993.

\bibitem {grub}W. Greub, \textit{Linear Algebra. 3rd ed.}, Springer Verlag,
Heidelberg, 1967.

\bibitem {simlorentz}Y. Kamishima, \textit{Lorentzian similarity manifolds,
}Cent. Eur. J. Math., 10(5), 1771-1788, 2012.
\end{thebibliography}
\end{document}